\documentclass{amsart}%
\usepackage{amsfonts}%
\usepackage{amsmath}%
\usepackage{amssymb}%
\usepackage{graphicx}
\usepackage{hyperref}
\usepackage{xcolor}
\providecommand{\U}[1]{\protect\rule{.1in}{.1in}}
\newtheorem{theorem}{Theorem}
\theoremstyle{plain}

\newtheorem{corollary}{Corollary}

\newtheorem{definition}{Definition}
\newtheorem{example}{Example}

\newtheorem{lemma}{Lemma}

\newtheorem{proposition}{Proposition}
\newtheorem{remark}{Remark}

\newcommand{\Z}{{\mathbb Z}}
\newcommand{\Ec}{\operatorname{EC}}
\newcommand{\Irr}{\operatorname{Irr}}
\newcommand{\Max}{\operatorname{Max}}
\newcommand{\ann}{\operatorname{ann}}
\numberwithin{equation}{section}
\begin{document}
\title[{\normalsize Quasi divisor topology of modules over domains}]{{\normalsize Quasi divisor topology of modules over domains}}
\author{Mesut Bu\u{g}day}
\address{Department of Mathematics, Marmara University, Istanbul, Turkey. }
\address{Department of Mathematics, Yildiz Technical University, Istanbul, Turkey}
\email{mesut.bugday@marmara.edu.tr, mesut.bugday@std.yildiz.edu.tr}
\author{Dilara Erdemir}
\address{Department of Mathematics, Yildiz Technical University, Istanbul, Turkey}
\email{dilaraer@yildiz.edu.tr}
\author{\"{U}nsal Tekir}
\address{Department of Mathematics, Marmara University, Istanbul, Turkey. }
\email{utekir@marmara.edu.tr}
\author{Suat Ko\c{c}}
\address{Department of Mathematics, Marmara University, Istanbul, Turkey. }
\email{suat.koc@marmara.edu.tr}

\begin{abstract}
Let $E$ be a module over a domain $A$, and $W(E)^{\#}=W(E)-ann(E)$ where $W(E)=\{a\in A:aE\neq
E\}$. We define an equivalence relation $\sim$ on $W(E)^{\#}$ as follows: $a\sim b$ if and only if $aE=bE$ for any $a,b\in W(E)^{\#}$ and denote $\Ec(W(E)^{\#})$ to be the set of all equivalence classes $[a]$ of $W(E)^{\#}$. We first show that the family $\{U_a\}_{a\in W(E)^\#}$ generates a topology which we called the quasi divisor topology of $A$-module $E$ denoted by $qD_A(E)$ where $U_{a}=\{[b]\in \Ec(W(E)^{\#}):\ aE\subseteq bE\}$ for every $a\in W(E)^{\#}$. This paper examines the connections between
topological properties of the quasi divisor topology $qD_{A}(E)$ and algebraic properties of $A$-module $E$. 
These include each separation axioms, compactness, connectedness and first and second countability. Also, we characterize some important class of rings/modules such as divisible modules and uniserial modules by means of $qD_{A}(E)$. Furthermore, we introduce quasi second modules and study its algebraic properties to decide when $qD_A(E)$ is a $T_1$-space.
\end{abstract}
\maketitle

\section{Introduction}
One of the popular applications of commutative algebra is to construct new topologies on algebraic structures and study from the point of view of algebra and topology. Oscar Zariski constructed such a prominent example in the literature, called the Zariski topology defined on varieties of a commutative ring. Moreover, one of the essences of the Zariski topology is that it leads to a better understading of a commutative ring in terms of geometry. Furthermore, the reader may consult \cite{Faranak2},\cite{Lu2},\cite{McCasland},\cite{Tekir},\cite{Eda} and \cite{secil} for more information about the Zariski topology with its further discussions and for other such constructions. Recently, the fourth author \cite{YiKo} constructed a new topology which is called the divisor topology over domains. In addition to proving theoretical results, they also provide an alternative way to prove that $\Z$ has infinitely many prime numbers.

In this note, we focus only on integral domains and modules over domains. Let $A$ always denote a domain and $E$ be an $A$-module. The annihilator of $E$ is denoted by $\ann(E)=\{a\in A:aE=0\}$, and
the set of all $a\in A$ such that the homothety $E\overset{a}{\rightarrow}E$
is not surjective, or equivalently $aE\neq E,$ is denoted by $W(E)$
\cite{Yas}. In fact, $W(E)\ $is the dual notion of the zero divisors of $A$ on
$E,\ z(E)=\{a\in A:\exists\ 0\neq e\in E$ such that $ae=0\},\ $and it has been
used in many papers to describe some important algebraic notions such as
second modules, cozero divisor graph of modules and etc. (See, \cite{coz}, \cite{coz2} and \cite{Yas}). It is clear that $\ann(E)\subseteq W(E)\ $and the
equality does not hold in general. For instance, take $A=%
\mathbb{Z}
$ and $E=%
\mathbb{Z}
_{pq}$  where $p\neq q$ are prime numbers, then $ann(E)=pq%
\mathbb{Z}
\subsetneq W(E)=p%
\mathbb{Z}
\cup q%
\mathbb{Z}
.\ $ 
Moreover, recall from \cite{Yas} that an $A$-module $E$ is called a \textit{second submodule} of itself or just a \textit{second module} if $aE=E$ or $aE=0$ for all $a\in A$. In fact, $L\ $is a second submodule if and only if $\ann(L)=W(L)=p$ is a prime ideal of $A$ (See, \cite[Lemma 1.2]{Yas}). In this paper, we construct a new topology for any $A$-module $E$ as follows: Let $W(E)^{\#}=W(E)-\ann(E)$ and define a relation $\sim$ on $W(E)^{\#}$ as for any $a,b\in W(E)^{\#},\ a\sim b$ if and only if $aE=bE$. Here, we note that $\sim$ is an equivalence relation on $W(E)^{\#}$ and denote the set of all equivalence classes $[a]$ of $W(E)^{\#}$ by $\Ec(W(E)^{\#})$, where $[a]$ denotes the equivalence class of $a$ with respect to the relation $\sim$. We first prove our first result in Proposition \ref{Prop1} that the family $\left\{U_{a}\right\}  _{a\in W(E)^{\#}}$ is a basis for the topology on $\Ec(W(E)^{\#})$ which is called the quasi divisor topology of $A$-module $E$ denoted by $qD_A(E)$ where $U_a$ is defined  for every $a\in W(E)^{\#}$ as $U_{a}=\{[b]\in \Ec(W(E)^{\#}):\ aE\subseteq bE\}$.
We deeply study not only various properties of $qD_A(E)$ but also seek to find possible relations between quasi divisor topology and algebraic structures of $A$-module $E$. More explicitly, this paper is organized as follows. In section $2$, we focus on fundamental properties of quasi divisor topology. For example, we identify all isolated points (See, Theorem \ref{isolated}) and also show that $qD_A(E)$ is an Alexandrov space (See, Corollary \ref{smallest}). Moreover, we characterize uniserial multiplication and torsion free divisible modules by properties of $qD_A(E)$ (See, Theorem \ref{nested} and Proposition \ref{noetherian}). We afterward define the closure and interior of a subset of $qD_A(E)$ in Proposition \ref{Closure}. Based on this definitions together with additional conditions, we prove some results regarding dense sets in $qD_A(E)$, see Theorem \ref{Irr(W(E)isdense}, Proposition \ref{opendenseset} and Corollary \ref{EC(S)dense}. In particular, we show that $qD_A(E)$ is a Baire space over a factorial module (See, Corollary \ref{Baire}). Furthermore, we show that quasi divisor topology of $A$-module $E$ where $E$ is a factorial module is homeomorphic to the divisor topology over a domain in Theorem \ref{homeo} together with \ref{homeo2}.

Next, section $3$ is devoted to investigate all separation axioms $T_i$ for $0\le i\le 5$. We show that $qD_A(E)$ is always a $T_0$-space (See, Proposition \ref{T0}), however not always a $T_1$-space as observed in Example \ref{Z_8isnotT_1}. For this reason, we tackle the following question. When $qD_A(E)$ is a $T_1$-space? We positively answered this question in Theorem \ref{T1-starcondition} by introducing quasi second modules. An $A$-module $E$ is said to be a \textit{quasi second module} if $0\neq bE\subseteq aE\neq E$ implies that $bE=aE$ for each $a,b\in A$. In particular, a commutative ring $A$ with unity is called a \textit{quasi second ring} if $A$-module $A$ is a quasi second module. One can see that vector spaces, divisible modules and second modules are examples of quasi second modules. Moreover, we examine the behavior of quasi second modules under module homomorphism, in factor module, in direct product of modules, in localization, and lastly the idealization of module together with related its properties (See, Theorem \ref{Homo qsm}, \ref{localization}, \ref{ApplyInduction}, \ref{qsr}, \ref{idealization} and Propositions \ref{annM semiprime}, \ref{ann(aE)maximal}). Explicitly, we prove in Theorem \ref{qsr} that $A$ is a quasi second ring if and only if either $(A,\mathfrak{m})$ is a local ring with $\mathfrak{m}^2=0$ or $A$ is a direct product of two fields. Moreover, we have results for finitely generated comultiplication modules, see Corollary \ref{ann(aE)maxforComultiplication}, Theorem \ref{SimpleorWeakIdempotent} and \ref{||=2}.

In the last section, we investigate other properties of quasi divisor topology, including Lindelöfness, compactness, connectedness, ultraconnectedness and hyperconnectedness. Moreover, as an application, we prove that if $E$ is a factorial module over an integral domain $A$ which is not a field with the condition $|S|<|A|$ where $S=\{a\in R:aE \text{ is a maximal element of the set } \{bE\}_{b\in W(E)^\#}\}$. Then, $A$ has infinitely many irreducible elements. In particular, we provide an alternative way to show the infinitude of irreducible elements of the localization of $\Z$ at the prime ideal $P$ by means of $qD_A(E)$.
\section{Quasi Divisor Topology}
Throughout the paper, we focus only modules over domains. Let $A$ always denote a domain, and $E$ be an $A$-module. We first recall the definitions from \cite{Lombardi} used in the following proposition. An integral domain $A$ is called \textit{Bézout} if the sum of two principal ideals is principal. Equivalently, finitely generated ideals are principal. Also, a domain is called a \textit{GCD-domain} if gcd exists for any nonzero pair of elements of $A$. Evidently, a Bezout domain is a GCD-domain.
\begin{proposition}\label{Prop1}
Let $E\ $be a module over domain $A.\ $Then the following statements are satisfied.

(i) $[a]\in U_a$ for any $a\in W(E)^\#$.

(ii) $U_a \subseteq U_b$ if and only if $bE\subseteq aE$. In particular, if $a\mid b$ then $U_a\subseteq U_b$.

(iii) $\cup_{a\in W(E)^\#} U_a=\Ec(W(E)^\#)$

(iv) If $[x]\in U_a\cap U_b$ for some $a,b\in W(E)^\#$, then $[x]\in U_x\subseteq U_a\cap U_b$.

(v) If $A$ is a Bezout domain so that $\gcd(a,b)\in W(E)^\#$ for any $a,b\in W(E)^\#$, then $U_{\gcd(a,b)}= U_a\cap U_b$.

(vi) Assume $A$ is a GCD-domain so that $\operatorname{lcm}(a,b)\in W(E)^\#$ for any $a,b\in W(E)^\#$. If $U_c\subseteq U_a$ and $U_c\subseteq U_b$ for some $a,b\in W(E)^\#$ then $U_c\subseteq U_{\operatorname{lcm}(a,b)}$.

(vii) Assume $E$ is a faithful multiplication module over a GCD-domain $A$ such that $\operatorname{lcm(a,b)}\in W(E)^\#$ for any $a,b\in W(E)^\#$. If $U_a\subseteq U_c$ and $U_b\subseteq U_c$ then $U_{\operatorname{lcm(a,b)}}\subseteq U_c$.
\end{proposition}
\begin{proof}
We show only the cases $(v),(vi)$ and $(vii)$ as the proof is similar to \cite[Proposition 1]{YiKo} for the rest. It is clear that $U_{\gcd(a,b)}\subseteq U_a\cap U_b$. Conversely, let $[x]\in U_a\cap U_b$ then $aE\subseteq xE$ and $bE\subseteq xE$. It follows that $aE+bE=((a)+(b))E=(\gcd(a,b))E\subseteq xE$, hence the claim follows.

For the proof of (vi), let $[x]\in U_c$ then $aE,bE\subseteq xE$ by assumption. Thus, $\operatorname{lcm}(a,b)E\subseteq aE,bE\subseteq xE$, as desired. 

For (vii), note that $cE\subseteq aE$ and $cE\subseteq bE$. Moreover, $cE\subseteq aE\cap bE=((a)\cap(b))E$ by \cite[Theorem 1.6]{Smith2}. Therefore, $((a)\cap(b))E=\operatorname{lcm}(a,b)E$. Thus, $U_{\operatorname{lcm}}(a,b) \subseteq U_c$.
\end{proof}
The aforementioned proposition indicates that $\{U_a\}_{a\in W(E)^\#}$ forms a basis for a topology on $\Ec(W(E)^\#)$ which is called the quasi divisor topology of $A$-module $E$, denoted by $qD_A(E)$. In the following example, we provide a counterexample for the opposite direction of Proposition \ref{Prop1} $(ii)$ to show that it is not valid in general.
\begin{example}
Consider the $\Z$-module $\Z_6$. One can see that $2\Z_6=\{\bar0,\bar2,\bar4\}=4\Z_6$ and so $U_4=U_2$ but $4\nmid 2$.
\end{example}
\begin{remark}
We note here that if module $E$ is chosen to be a second module then $W(E)=\ann(E)$ which gives that $W(E)^\#=\Ec(W(E)^\#)=\emptyset$ and so $qD_A (E)$ will be an empty space. Due to this, we do not consider module $E$ to be a second module unless otherwise noted.
\end{remark}
\begin{example}
Consider the $\Z$-module $\Z$. One can determine that $\ann(\Z)=0$ and $W(\Z)=\{a\in \Z:a\neq \pm1\}$. Moreover, $U_a=\{b\in W(\Z)^\#: \text{ $a$ is an integer multiple of $b$}\}$. Hence, $qD_\Z (\Z)$ corresponds to the divisor topology on integers as in \cite{YiKo}.
\end{example}
\begin{lemma}\label{smallest}
Let $E$ be a module over a ring $A$ and $a\in W(E)^\#$. Then, $U_a$ is the smallest open set containing $[a]$.
\end{lemma}
\begin{proof}
Suppose $U_a$ is not the smallest open set containing $[a]$, then there exists an open set $O$ such that $[a]\in O=\bigcup_{x\in \wedge} U_x\subseteq U_a$ for some index set $\wedge$. This implies that there is an index $x$ so that $[a]\in U_x$. Hence, $[a]\in U_x\subseteq \bigcup_{x\in \wedge} U_x\subseteq U_a$ which implies that $aE=xE$ and so $U_a=U_x$. Therefore, the open set $O=U_a$, which completes the proof.
\end{proof}
Recall from \cite{Arenas} that a topological space is called \textit{Alexandrov} if an arbitrary intersection of open sets is open, which is equivalent to the fact that each $x\in X$ has a minimal neighborhood. The following corollary is a direct consequence of Lemma \ref{smallest}.
\begin{corollary}
$qD_A(E)$ is an Alexandrov space.
\end{corollary}
In the next theorem, we describe all isolated points of $qD_A(E)$. Recall that an element $x\in X$ is called an \textit{isolated point} if the singleton $\{x\}$ forms an open set, \cite{Munkres}.
\begin{theorem}\label{isolated} Let $E$ be an $A$-module. Then the following statements are equivalent:
\begin{enumerate}
\item[$i)$] $[a]\in \Ec(W(E)^\#)$ is an isolated point.
\item[$ii)$] $U_a=\{[a]\}$ for all $a\in W(E)^\#$.
\item[$iii)$] $aE$ is a maximal element of the set $\{bE:b\in W(E)^\#\}$ for all $a\in W(E)^\#$.
\end{enumerate}
\end{theorem}
\begin{proof}
$i)\Leftrightarrow ii):$ This follows from Lemma \ref{smallest}.

$ii)\Leftrightarrow iii):$ This follows from the maximality of $aE$ in the set $\{bE\}_{b\in W(E)^\#}$. 
\end{proof}
We now recall the necessary definitions and notations that will be used in the following discussion which can be found in \cite{Lu}. Let $E$ be an $A$-module. We first start with the notations $|_A$ and $|_E$ to denote division in $A$ and $E$, respectively. An element $m\in E$ is called \textit{irreducible} in $E$ if it has no proper factor in $E$. Moreover, an element $a\in A$ is called prime to $E$ if $a$ is irreducible in $A$ and if $a\mid bm$ then either $a\mid_A b$ or $a\mid_E m$. A torsion free $A$-module $E$ is called a \textit{factorial module} if every nonzero element $m\in E$ has a unique irreducible factorization up to order and associates such that $m=a_1\dots a_nm'$ where $a_1,\dots,a_n$ are irreducible in $A$ and $m'$ is irreducible in $E$. We also note here that irreducible elements of $A$ are prime to $E$ and a ring $A$ is necessarily a unique factorization domain if $E$ is a factorial $A$-module. 

\begin{proposition}\label{primetoMandmaximality}
Let $E$ be a torsion free $A$-module. If $A$ has an element which is prime to $E$ then $aE$ is a maximal element of the set $\{bE:b\in W(E)^\#\}$.
\end{proposition}
\begin{proof}
Suppose $a\in A$ is an element which is prime to $E$ and $aE\subsetneq bE$ for some $b\in W(E)^\#$. It follows that for any $m\in E$, there is an $m'\in E$ such that $am=bm'$. Since $a\in A$ is prime to $E$, we have $a\mid_A b$ or $a\mid_E m'$. However, the first case implies $aE=bE$ and the latter implies $bE=E$, which both are contradictions. Hence, $aE$ is a maximal element of the set $\{bE\}_{b\in W(E)^\#}$. 
\end{proof}
\begin{corollary}
Let $E$ be a multiplication module over a principal ideal domain $A$. Then, $[a]\in \Ec(W(E)^\#)$ is an isolated point if and only if $aE$ is a maximal submodule of $E$.
\end{corollary}
\begin{proof}
One direction of the statement is straightforward by Theorem \ref{isolated}. Conversely, each submodule of $E$ is of the form $IE=(a)E=aE$ for some ideal $I$ of $A$ since $E$ is a multiplication $A$-module. Now, if there is a proper submodule $N$ such that $aE\subseteq N=(b)E=bE$ for some $b\in A$, then $[b]\in U_a$ which implies $[b]=[a]$. Hence, $aE$ is a maximal submodule of $E$.
\end{proof}
A natural question can be posed that is there a connection between the algebraic properties of module $E$ and the topological properties of $qD_A(E)$. We now briefly discuss the following results closely related to this problem.

A topological space $X$ is called \textit{nested} if all open sets of $X$ are linearly ordered by inclusion \cite{Richmond}. Moreover, a topology $\tau$ on $X$ with a basis $B$ is called nested if and only if for every $B_1,B_2 \in B$ either $B_1\subseteq B_2$ or $B_2\subseteq B_1$, \cite{YiKo}. In the following theorem, we give a characterization of nested space for $qD_A(E)$ using uniserial modules. Recall that an $A$-module $E$ is called an \textit{uniserial module} if each submodule of $E$ is comparable with respect to inclusion, \cite{FacSal}.
\begin{theorem}\label{nested}
Let $E$ be a multiplication $A$-module. Then, $E$ is an uniserial module if and only if $qD_A(E)$ is a nested topology.
\end{theorem}
\begin{proof}
Suppose $E$ is an uniserial module then either $aE\subseteq bE$ or $bE\subseteq aE$. Hence, $U_b\subseteq U_a$ or $U_a\subseteq U_b$. Conversely, let $N$ and $K$ be two submodules of $E$ such that $N\not\subseteq K$. Since $E$ is a multiplication module, we have $N=(N:E)E$ and $K=(K:E)E$. It follows that there is $n\in N\setminus K$ such that $n=am$ for some $a\in (N:E)$ and $a\not\in (K:E)$. That means $aE\not\subseteq bE$ for all $b\in(K:E)$. Moreover, since $qD_A(E)$ is a nested topology, we must have $bE\subseteq aE$ for all $b\in (K:E)$. Therefore, $K\subseteq aE\subseteq N$, which completes the proof.
\end{proof}
Recall from \cite{Hartshorne} that a topological space $X$ is called \textit{Noetherian} if any chain of closed subsets of $X$ satisfies the descending chain condition. This is equivalent to the fact that there is no strictly increasing infinite chain of open subsets of $X$.
\begin{proposition}\label{noetherian}
Let $E$ be a torsion free $A$-module. Then, $E$ is a divisible $A$-module if and only if $qD_A(E)$ is a Noetherian space.
\end{proposition}
\begin{proof}
Let $a$ be a nonzero element of $A$ and consider the following infinite chain of open sets $$U_a\subseteq U_{a^2}\subseteq \dots \subseteq U_{a^n}\subseteq U_{a^{n+1}}\subseteq \dots$$ Since $qD_A(E)$ is a Noetherian space, the chain must eventually stop i.e., there exists $k\in \mathbb N$ such that $U_{a^k}=U_{a^{k+1}}=\ldots$; so $a^kE=a^{k+1}E$. Hence, for any $m\in E$ there is an $m'\in E$ such that $a^{k+1}m'=a^km$. It follows that $m=am'$ and so $E=aE$, which completes the proof. The other direction is straightforward.
\end{proof}
In the following proposition, we define the closure and interior of a given subset of $qD_A(E)$.
\begin{proposition}\label{Closure}
Let $E$ be an $A$-module, and $N\subseteq qD_A(E)$. Then, $$\overline{N}=\{[a]\in \Ec(W(E)^\#): U_b\subseteq U_a \text{ for }[b]\in N\} \text{ and } N^\circ=\{[a]\in N: U_a\subseteq N\}.$$
In particular, $\overline{\{[a]\}}=\{[b]\in \Ec(W(E)^\#): bE\subseteq aE\}$ for $[a]\in \Ec(W(E)^\#)$.
\end{proposition}
\begin{proof}
Let $[a]\in \overline{N}$ for some $a\in W(E)^\#$ then any open set containing $[a]$ intersects with $N$. Thus, there exists $[b]\in N$ such that $U_b\subseteq U_a$. Conversely, suppose $U_b\subseteq U_a$ for some $[b]\in N$, then $[b]\in U_a$. It follows that $[b]\in N\cap U_a$ and so $U_a\cap N\neq \emptyset$. This implies that $[a]\in \overline{N}$, which completes the first part of the proof.

On the other hand, let $[b]\in N^\circ$ then there is an open set $U\subseteq N$ such that $[b]\in U$. Thus, $[b]\in U_b\subseteq U\subseteq N$ by Lemma \ref{smallest}. Conversely, let $[b]\in \{[a]\in N: U_a\subseteq N\}$ then $U_b\subseteq N$ which immediately implies $U_b\subseteq N^\circ$, which concludes the proof.
\end{proof}
\begin{lemma}\label{unionofclosure}
Let $\{N_i\}_{i\in I}$ be a family of subsets of $qD_A(E)$. Then, $$\overline{\bigcup_{i\in I} N_i}=\bigcup_{i\in I} \overline{N_i}.$$
\end{lemma}
\begin{proof}
Let $[a]\in \overline{\bigcup_{i\in I} N_i}$ then $ [a]\in N_{i_0}$ for some $i_0\in I$ by Proposition \ref{Closure}. It follows that $U_b\subseteq U_a$ for some $[b]\in \cup_{i\in I} N_{i}$. Thus, $[a]\in \bigcup_{i\in I} \overline{N_i}$. The converse part can be shown similarly by following up the reverse directions.
\end{proof}
\begin{theorem}\label{Irr(W(E)isdense}
Let $E$ be a module over a factorization domain $A$ with the property that $\Irr(A)\subseteq W(E)^\#$ then $\Ec(\Irr(W(E)^\#))$ is dense in $qD_A(E)$.
\end{theorem}
\begin{proof}
We only need to verify that $\overline{\Ec(\Irr(W(E)^\#))}\supseteq\Ec(W(E)^\#)$. Let $a\in W(E)^\#$ then the irreducible factorization of $a=r_1\dots r_n$ yields that $aE\subseteq r_1 E$ where $r_1,\dots,r_n$ are irreducible in $A$. Thus, $[a]\in \overline{\{[r_1]\}}\in \overline{\Ec(\Irr(W(E)^\#))}$. 
\end{proof}
Before giving the following proposition and corollary, we set $S=\{a\in R:aE \text{ is a maximal element of the set } \{bE\}_{b\in W(E)^\#}\}$
\begin{proposition}\label{opendenseset}
Let $E$ be an $A$-module. Then each open dense set in $qD_A(E)$ contains $\Ec(S)$.
\end{proposition}
\begin{proof}
Let $O$ be an open dense set in $qD_A(E)$ then $\overline{O}=\Ec(W(E)^\#)$. Let $a\in S$ then $[a]\in \Ec(W(E)^\#)=\overline{O}$. It follows that any open set containing $[a]$, particularly $U_a$, intersects with $O$. In this case, $U_a=\{[a]\}$ as $a\in S$ which implies that $[a]\in O$. Therefore, $\Ec(S)\subseteq O$.
\end{proof}
\begin{corollary}\label{EC(S)dense}
Let $E$ be a factorial $A$-module then $\Ec(S)$ is dense in $qD_A(E)$.
\end{corollary}
\begin{proof}
The proof is similar to Theorem \ref{Irr(W(E)isdense} and justified by the fact that each irreducible element of $A$ is prime to $E$ by \cite{Lu}.
\end{proof}
A topological space $X$ is called a \textit{Baire} space if the union of countably many closed sets has an empty interior in $X$ in which each of the closed sets in the intersection has an empty interior; equivalently, the intersection of countably many open sets is dense in $X$ in which each of the open sets in the intersection is dense in $X$ \cite[Theorem 48.1]{Munkres}. Now, we will determine when $qD_A(E)$ is a Baire space.
\begin{corollary}\label{Baire}
Let $E$ be a factorial module over a ring $A$ then $qD_A(E)$ is a Baire space.
\end{corollary}
\begin{proof}
The proof is analogous to \cite[Theorem 4]{YiKo} and holds in view of Proposition \ref{opendenseset}, Corollary \ref{EC(S)dense} and \cite[Theorem 48.1]{Munkres}.
\end{proof}
We end up this section by providing two results regarding homeomorphism. A \textit{homeomorphism} between two topological spaces $X$ and $Y$ is a function $f:X\to Y$ that is a continuous bijection with a continuous inverse $f^{-1}:Y\to X$. In case of the existence of such a homemorphism between these topological spaces, we say $X$ and $Y$ are homeomorphic \cite{Munkres}.
\begin{theorem}\label{homeo}
The quasi divisor topology $qD_A(E)$ over a factorial module $E$ is homeomorphic to the divisor topology on integral domains $D(A)$. In particular, the topologies are the same.
\end{theorem}
\begin{proof}
The proof follows from the fact that the equivalence relation $\sim$ and defined in \cite{YiKo} are the same in factorial module.
\end{proof}
\begin{theorem}\label{homeo2}
Let $E$ be a finitely generated multiplication module over an integral domain $A$. If $\ann(E)$ is a prime ideal, then the divisor topology on $A/\ann(E)$ and $qD_A(E)$ are homeomorphic.
\end{theorem}
\begin{proof}
One may check that the following claimed function
$\varphi:\Ec((A/\ann(E) ) ^\#)\to \Ec(W(E)^\#)$ defined by $[r+\ann(E)]\mapsto [r]$ satisfies the homeomorphism conditions as a result of the properties that $\varphi(B_{a+\ann(E)})=U_a$ and $\varphi^{-1}(U_a)=B_{a+\ann(E)}$ by \cite[\S 2 Corollary]{Smith} where the basis element $B_a$ is defined in \cite{YiKo}.
\end{proof}
\section{Separation Axioms and Quasi Second Modules}
It is a classic routine to check the separability axioms for a given topology. Recall that a topological space $X$ is called a \textit{$T_0$-space} if for any two distinct points in $X$, there is an open set containing not both of them. Also, $X$ is called a \textit{$T_1$-space} if given two distinct points of $X$, there is an open set for each containing one not the other; $X$ is $T_1$-space if and only if each singleton is closed. Moreover, $X$ is called a \textit{Hausdorff} or \textit{$T_2$-space} if for any two distinct points of $X$, there are disjoint open sets containing them, \cite{Munkres}.
\begin{proposition}\label{T0}
$qD_A(E)$ is a $T_0$-space.
\end{proposition}
\begin{proof}
Given two distinct elements $[a]$ and $[b]$ in $\Ec(W(E)^\#)$, we have $aE\neq bE$ which provides that either $[b]\not\in U_a$ or $[a]\not\in U_b$.
\end{proof}
The previous proposition indicates that $qD_A(E)$ is always a $T_0$-space; however, the following example shows that it is not the case for $T_1$-space.
\begin{example}\label{Z_8isnotT_1}
Consider the $\Z$-module $\Z_8$. One can compute that $EC(W(\Z)^\#)=\{[2],[4]=[6]\}$ together with the basis elements $U_2=\{[2]\}$ and $U_4=\{[2],[4]\}$. It is clear that $qD_A(E)$ is not a $T_1$-space because it is not possible to separate the points $[2]$ and $[4]$ by open sets.
\end{example}
The following question in this context can naturally arise whether which type of modules enable $qD_A(E)$ to be a $T_1$-space. In order to answer this question, we introduce quasi second modules motivated by this problem as follows.
\begin{definition}
An $A$-module $E$ is said to be a \textit{quasi second module} if $0\neq bE\subseteq aE\neq E$ implies that $bE=aE$ for each $a,b\in A$. In particular, a commutative ring $A$ with unity is called a \textit{quasi second ring} if $A$-module $A$ is a quasi second module.
\end{definition}
Moreover, an $A$-module $E$ is called a \textit{divisible module} if $aE=E$ for all $a\in A$. Now, we give several examples of quasi second modules.
\begin{example}\label{Example2}
$$\{\text{Vector Spaces}\}\subsetneq \{\text{Divisible Modules}\}\subsetneq\{\text{Second Modules}\}\subsetneq\{\text{Quasi Second Modules}\}$$
\end{example}
In the following example, we provide counterexamples for each opposite direction of Example \ref{Example2}.
\begin{example}\label{Example3}
\begin{enumerate}
\item[i)] $\Z$-module $\mathbb Q$ is a divisible module which is not a vector space.
\item[ii)] $\Z$-module $\Z_p$ is a second module which is not a divisible module where $p$ is a prime number.
\item[iii)] $\Z$-module $\Z_{p^2}$ is a quasi second module but not a second module.
\end{enumerate}
\end{example}
\begin{theorem}\label{Homo qsm}
Let $\varphi:E\to E'$ be a $A$-module homomorphism. Then,
\begin{enumerate}
\item[i)] $E$ is a quasi second $A$-module if $\varphi$ is injective and $E'$ is quasi second $A$-module.
\item[ii)] $E'$ is a quasi second $A$-module if $\varphi$ is surjective and $E$ is a quasi second $A$-module.
\end{enumerate}
\end{theorem}
\begin{proof}
\begin{enumerate}
\item[i)] Let $a,b\in A$ such that $0\neq aE\subseteq bE\neq E$. Note that $aE\neq 0$ and $bE'\neq E'$ as $\varphi$ is injective. Thus, $0\neq aE'\subseteq bE'\neq E'$ which implies that $aE=bE$.
\item[ii)] Let $a,b\in A$ so that $0\neq aE'\subseteq bE'\subseteq E'$. Since $\varphi$ is surjective, then $aE\neq 0$ and $bE\neq E$. It follows that $0\neq aE\subseteq bE\neq E$ and so $aE'=bE'$.
\end{enumerate}
\end{proof}
The next corollary can be obtained by Theorem \ref{Homo qsm}.
\begin{corollary}\label{ExamplesofQSM}
i) Every submodule of a quasi second module is a quasi second module.

ii) Any factor module by a submodule of a quasi second module is a quasi second module. In particular, $E$ is a quasi second $A$-module if and only if $E/E'$ is a quasi second $A$-module for all submodule $E'$ of $E$.

iii) If $\oplus_i E_i$ or $\prod_i E_i$ is a quasi second $A$-module for the collection of $A$-modules $\{E_i\}_{i\in I}$ then each $E_i$ is a quasi second $A$-module.
\end{corollary}
\begin{theorem}\label{localization}
Let $E$ be an $A$-module with a multiplicative closed set $S$ of $A$. Then, $S^{-1}E$ is a quasi second $S^{-1}A$-module if $E$ is a quasi second $A$-module.
\end{theorem}
\begin{proof}
Let $\frac{a}{s}$ and $\frac{b}{s'}$ be elements of $S^{-1}A$ such that $0_{S^{-1}E}\neq \frac{a}{s}S^{-1}E\subseteq \frac{b}{s'}S^{-1}E\neq S^{-1}E$. We claim that $0\neq aE\subseteq bE\neq E$. Suppose that the claim does not hold, then there is an $m\in E$ such that $am\not\in bE$. It follows that $\frac{am}{ss'}\in \frac{a}{s}S^{-1}E\subseteq \frac{b}{s}S^{-1}E$. This implies that there exists $t\in S$ so that $ts(ams-bm's'')=0$ for some $m'\in E$ and $s''\in S$. Hence, $\frac{am}{s'}=\frac{bm'}{s''}\in S^{-1}(bE)$, a contradiction which completes the proof.
\end{proof}
In the next example, we provide a counterexample for the opposite direction of Theorem \ref{localization}.
\begin{example}
Consider the $\Z$-module $E=\Z^2$ with a multiplicative subset $S=\Z-\{0\}$. One can see that $A$ is not a quasi second ring but $S^{-1}E$ is a quasi second $A$-module where $S^{-1}A=\mathbb Q$ is the total quotient field of $\Z$ and $S^{-1}E\cong \mathbb Q\times \mathbb Q$.
\end{example}

Recall that the proper ideal $I$ of a ring $A$ is called \textit{semiprime} (or \textit{radical}) ideal if whenever $a^n\in I$ for some $a\in A$ and $n\in \mathbb N$ implies that $a\in I$. It is well known that an ideal $I$ is a semiprime ideal if and only if for every $a\in A$, $a^2\in I$ implies that $a\in I$  if and only if $\sqrt{I}=I$. 
\begin{proposition}\label{annM semiprime}
Suppose $E$ is a quasi second $A$-module then $0\neq aE\neq E$ is a second submodule of $E$ for every $a\in A$. In particular, the converse holds if $\ann(E)$ is a semiprime ideal of $A$.
\end{proposition}
\begin{proof}
Suppose $E$ is a quasi second $A$-module and choose $a\in A$ such that $0\neq aE\neq E$. Let $x\in A$ so that $0\neq xaE\subseteq aE$. It follows that $xaE=aE$ by the assumption, thus $aE$ is a second submodule of $E$. For the converse, assume $\ann(E)$ is a semiprime ideal of $A$ and $0\neq aE\neq E$ is a second submodule of $E$ for every $a\in A$. Let $0\neq bE\subseteq aE$ then either $abE=aE$ or $abE=0$ as $aE$ is a second submodule. If $abE=0$ then $b^2E=0$ which implies $bE=0$ since $\ann(E)$ is a semiprime ideal. Therefore, $abE\neq 0$ which gives $aE=abE\subseteq bE\subseteq aE$, that is, $aE=bE$. Thus, $E$ is a quasi second $A$-module.
\end{proof}
\begin{proposition}\label{ann(aE)maximal}
Let $E$ be a finitely generated quasi second $A$-module then $\ann(aE)$ is a maximal ideal of $A$ for every $a\in W(E)^\#$.
\end{proposition}
\begin{proof}
Suppose $x\not\in\ann(aE)$ for some $x\in A$ and $a\in W(E)^\#$. Then, $0\neq xaE\subseteq aE\neq E$ which implies that $xaE=aE$ as $E$ is a quasi second $A$-module. Moreover, since $E$ is finitely generated, so is $aE$. It follows that there exists $r\in xA$ such that $(1-r)aE=0$ by Nakayama's Lemma. Hence $1-r\in \ann(aE)$ and so $xA+\ann(aE)=A$. Thus, $\ann(aE)\in Max(A)$.
\end{proof}
Recall from \cite{Faranak} that $A$-module $E$ is said to be a \textit{comultiplication module} if for every submodule $N$ of $E$ there exists an ideal $I$ of $A$ such that $N=\ann_E(I)$. It is well known that $E$ is a comultiplication module if and only if $N=\ann_E(\ann_A(N))$ for every submodule $N$ of $E$.
\begin{theorem}\label{Theorem1}
Let $E$ be an $A$-module then the following statements hold:
\begin{enumerate}
\item[i)] If $\ann(E)$ is a semiprime ideal and $\ann(aE)\in\Max(A)$ for every $a\in W(E)^\#$, then $E$ is a quasi second $A$-module.
\item[ii)] If $E$ is a comultiplication $A$-module and $\ann(aE)\in\Max(A)$ for all $a\in W(E)^\#$, then $E$ is a quasi second $A$-module.
\end{enumerate}
\end{theorem}
\begin{proof}
\begin{enumerate}
\item[i)] The proof holds in view of Proposition \ref{annM semiprime}. Let $x\in A$ such that $xaE\neq 0$ then $x\not\in \ann(aE)$ and $xA+\ann(aE)=A$ by the maximality of $\ann(aE)$. This implies that $xaE+\ann(aE)aE=xaE=aE$. Therefore, $aE$ is a second submodule of $E$.

\item[ii)] Let $0\neq bE\subseteq aE\neq E$ for some $b\in A$. This gives $\ann(aE)\subseteq\ann(bE)\neq A$. Since $E$ is a comultiplication module and $\ann(aE)\in \Max(A)$, we get $aE=\ann_E(\ann_A(aE))=\ann_E(\ann_A(bE))=bE$, which completes the proof.
\end{enumerate}
\end{proof}
\begin{corollary}
Let $E$ be a finitely generated $A$-module such that $\ann(E)$ is a semiprime ideal. Then the following statements are equivalent:
\begin{enumerate}
\item[i)] $E$ is a quasi second $A$-module.
\item[ii)] $aE$ is a second submodule of $E$ for every $a\in W(E)^\#$.
\item[iii)] $\ann(aE)\in\Max(A)$ for every $a\in W(E)^\#$.
\end{enumerate}
\end{corollary}
\begin{proof}
The proof follows from Proposition \ref{annM semiprime}, \ref{ann(aE)maximal} and Theorem \ref{Theorem1} i).
\end{proof}
\begin{corollary}\label{ann(aE)maxforComultiplication}
Let $E$ be a finitely generated comultiplication $A$-module. Then the following statements are equivalent:
\begin{enumerate}
    \item[$i)$] $E$ is a quasi second $A$-module.
    \item[$ii)$] $aE$ is a second submodule of $E$ for every $a\in W(E)^\#$. 
    \item[$iii)$] $\ann(aE)\in\Max(A)$ for every $a\in W(E)^\#$.
\end{enumerate}
\end{corollary}
\begin{proof}
$(i)\Rightarrow (ii)$: Follows from Proposition \ref{annM semiprime}.

$(ii)\Rightarrow (iii)$: The proof is similar to Proposition \ref{ann(aE)maximal} applying Nakayama's lemma on the submodule $aE$.

$(iii)\Rightarrow (i)$: Follows from Theorem \ref{Theorem1} ii).
\end{proof}
\begin{theorem}\label{ApplyInduction}
Let $E_1$ and $E_2$ be two finitely generated $A$-module with the condition that $\ann(E)$ is a semiprime ideal and $W(E_1\oplus E_2)^\#\subseteq W(E_1)^\# \cap W(E_2)^\#$. Then, $E=E_1\oplus E_2$ is a quasi second $A$-module if and only if $E_1$ and $E_2$ are quasi second $A$-module and $\ann(aE_1)=\ann(aE_2)$ is a maximal ideal of $A$ for all $a\in W(E_1)^\# \cap W(E_2)^\#$.
\end{theorem}
\begin{proof}
$\Rightarrow:$ It sufficies to show that $\ann(aE_1)=\ann(aE_2)$ is a maximal ideal of $A$ by Corollary \ref{ExamplesofQSM} iii). First, note that $W(E_{1})^\#\subseteq W(E_1\oplus E_2)$ and let $a\in W(E_1)^\#$ then $\ann(a(E_1\oplus E_2))=\ann(aE_1)\cap \ann(aE_2) \in \Max(A)$. Therefore, $\ann(aE_1)=\ann(aE_2)$ is a maximal ideal, as desired.

$\Leftarrow:$ Suppose that there is $a\in W(E_1\oplus E_2)^\#\subseteq W(E_1)^\# \cap W(E_2)^\#$ such that $\ann(a(E_1\oplus E_2))$ is not a maximal ideal. However, $\ann(a(E_1\oplus E_2))=\ann(aE_1)$ by the assumption that leads to a contradiction.
\end{proof}
The following result can be proved by induction on $n$ within Theorem \ref{ApplyInduction}.
\begin{corollary}
Assume that $E_1,\dots, E_n$ are $A$-modules such that $\ann(E)$ is a semiprime ideal of $A$ with the property that $W(\oplus_{i=1}^n E_i)^\# \subseteq \bigcap_{i=1}^nW(E_i)^\#$.
\end{corollary}
Recall that an element $r\in A$ is called a \textit{weak idempotent} if $r^2-r\in ann(E)$ \cite{vn}.
\begin{theorem}\label{SimpleorWeakIdempotent}
Let $E$ be a finitely generated comultiplication module over $A$ such that $\ann(E)$ is a semiprime ideal. Then, $E$ is a quasi second $A$-module if and only if one of the following conditions hold:
\begin{enumerate}
    \item[i)] $E$ is a simple module.
    \item[ii)] $E=eE\oplus (1-e)E$ for some weak idempotent $e\in A$ where $eE$ and $(1-e)E$ are minimal submodules of $E$. In this case, $E=Am\oplus Am'$ is a semisimple module.
\end{enumerate}
\end{theorem}
\begin{proof}
$\Rightarrow:$ Suppose $E$ is a quasi second $A$-module and choose $a\in W(E)^\#$. Since $\ann(E)$ is a semiprime ideal, then $a^2E\neq 0$. It follows that since $E$ is a quasi second module and $0\neq a^2E\subseteq aE$, we get $a^2E=aE$. Moreover, since $a^2E=aE$ and $E$ is a finitely generated $A$-module then $E$ is vn-regular module by \cite[Theorem 1]{vn}. In addition, $aE=eE$ for some weak idempotent $e\in A$ by \cite[Lemma 5]{vn}, and the element $\bar e$ is an idempotent element of $A/\ann(E)$. Now, we will show that $aE$ is a minimal submodule of $E$. Let $0\neq K\subseteq aE$ for some submodule $K$ of $E$. Then $\ann(aE)\subseteq \ann(K)\neq A$. By Corollary \ref{ann(aE)maxforComultiplication}, we have $\ann(aE)=\ann(K)$. It follows that $aE=K$ since $E$ is a comultiplication module, as desired.

Let $\operatorname{ID}(A/\ann(E))$ denote the set of idempotent elements of $A/\ann(E)$ and assume that $\operatorname{ID}(A/\ann(E))=\{\bar0,\bar 1\}$. This implies that $xE=0$ or $xE=E$ for every $x\in A$. Moreover, $E$ is a simple module since $E$ is a vn-regular module and $|\operatorname{ID}(A/\ann(E))|=2$. Now, assume that $|\operatorname{ID}(A/\ann(E))|>2$, and choose $\bar0,\bar 1\neq \bar e\in \operatorname{ID}(A/\ann(E))$. Let $x\in W(E)^\#$, we will show that $xE=eE$ or $xE=(1-e)E$. First, one can see that $e\in W(E)^\#$ by Nakayama's Lemma. Now, assume that $xE\neq eE$. If $xeE\neq 0$ then $0\neq xeE\subseteq xE\cap eE$ implies that $xeE=xE=eE$ since $E$ is a quasi second module, which leads to a contradiction. Hence, we conclude that $xeE=0$ which gives $x\in \ann(eE)=(1-e)+\ann(E)$ by \cite{JTK}. Thus, $xE\subseteq (1-e)E$ and since $E$ is a quasi second $A$-module then $xE=(1-e)E$. Therefore, $E=eE\oplus (1-e)E$ by \cite{vn}. We finally note that since $eE$ and $(1-e)E$ are minimal (simple module) submodules, $E$ is semisimple.

$\Leftarrow:$ Assume that $E$ is a simple module then, $E$ is a quasi second $A$-module by Example \ref{Example2}. Now assume that ii) holds then $E=Am\oplus Am'$ where $Am=eE$ and $Am'=(1-e)E$ are minimal submodules for some weak idempotent $e\in A$. Let $a\in W(E)^\#$ and choose $b\in A$ such that $0\neq bE\subseteq aE$. Note that $bE=Abm\oplus Abm'$ and $aE=Aam\oplus Aam'$. Since $bE\subseteq aE$, one can write $bm=xam+yam'$ for some $x,y\in A$. Thus, $bm-xam=yam'\in Am\cap Am'=\{0\}$ and so $bm=xam$, and consequently $Abm\subseteq Aam\subseteq Am$. Similarly, we have $Abm'\subseteq Aam'\subseteq Am'$. Since $aE=Aam\oplus Aam'\neq E=Am\oplus Am'$ and $Am,Am'$ are simple modules, then either $Aam=0$ or $Aam'=0$. One may assume that $Aam\neq 0$ and $Aam'=0$. In that case, $Abm'=0$ and $Abm\neq 0$. Hence, one can conclude that $bE=Abm=Aam=aE$ as $Am$ is simple and $0\neq Abm\subseteq Aam \subseteq Am$.
\end{proof}
\begin{theorem}\label{qsr}
$A$ is a quasi second ring if and only if either $(A,\mathfrak{m})$ is a local ring with $\mathfrak{m}^2=0$ or $A$ is a direct product of two fields.
\end{theorem}
\begin{proof}
$\Leftarrow:$ \textbf{Case 1:} Assume that $(A,\mathfrak{m})$ is a local ring with $\mathfrak{m}^2=0$. Let $0\neq x\in A$ be a nonunit element and $0\neq Ay\subseteq Ax$ for some $y\in A$. It follows that $y=ax$ for some $a\in A$. If $a$ is a nonunit element, then $a,x\in \mathfrak{m}$ which implies that $y=ax=0$, a contradiction. Thus, $a$ is nonunit element and so $Ay=Ax$.

\textbf{Case 2:} Assume that $A=F_1\times F_2$ where $F_1$ and $F_2$ are two fields. Let $x=(x_1,x_2)\in A$ be a nonzero nonunit element of $A$. Then, one may assume that $x_1=0$ and $x_2\neq 0$. Now, choose $y=(y_1,y_2)\in A$ such that $0\neq Ay\subseteq Ax$. It follows that $y_1=0$ and $y_2\neq 0$ and thus $Ay=0\times F_2y_2=0\times F_2=0\times F_2x_2=Ax$.

$\Rightarrow:$ Let $A$ be a quasi second ring. First, we show that $|\operatorname{Spec}(A)|\le 2$. Let $P_1,P_2,P_3$ be prime ideals of $A$ and choose $x\in P_1-(P_2\cup P_3)$. Since $A$ is a quasi second ring, $\ann(x)$ is a maximal ideal by Proposition \ref{ann(aE)maximal}. Since $x\not \in P_2\cup P_3$, we have $\ann(x)\subseteq P_2$ and $\ann(x)\subseteq P_3$ which gives $\ann(x)=P_2=P_3$. Thus, $|\operatorname{Spec}(A)|\le 2$. Now, we have the following two cases.

\textbf{Case 1:} Suppose that $\operatorname{Spec}(A)=\{P\}$. We will show that $P^2=0$. If $P=0$, then there is nothing to prove. Assume $P\neq 0$ and choose $0\neq x\in P$, by above argument, we have $P=\ann(x)$ is a maximal ideal. Now, pick $0\neq y_1,y_2\in P=\ann(x)$, then $y_1x=0=y_2x$. On the other hand, since $A$ is a quasi second ring and $(A,P)$ is a local ring, $\ann(y_1)=P=\ann(y_2)=\ann(x)$. Since $y_2\in \ann(x)=\ann(y_1)$, we get $y_1y_2=0$. Thus, $P^2=0$, i.e., $(A,P)$ is a local ring with $P^2=0$.

\textbf{Case 2:} Suppose that $\operatorname{Spec}(A)=\{P_1,P_2\}$. Choose $x\in P_1-P_2$ and $y\in P_2-P_1$. Since $A$ is a quasi second ring, $\ann(x)=P_2$ and $\ann(y)=P_1$ are maximal ideals of $A$. Now, we will show that $P_1\cap P_2=\{0\}$. Choose $0\neq a\in P_1\cap P_2$, since $\ann(a)$ is a maximal ideal, one may assume that $\ann(a)=\ann(x)=P_2$. If $ax\neq 0$, then $0\neq Aax\subseteq Ax\neq A$ and $0\neq Aax\subseteq Aa\neq A$. This implies that $Aax=Ax=Aa$, however, this is a contradiction since $a\in P_2$ and $x\not \in P_2$. Thus, we get $ax=0$ and similarly $ay=0$. It follows that $a(x+y)=0$ and so $x+y\in \ann(a)=\ann(x)$. If $xy\neq 0$, then $0\neq Axy\subseteq Ax$. This similarly implies $Axy=Ax=Ay$, which is again a contradiction. Hence, $xy=0$. Now, since $(x+y)x=0=xy$, we have $x^2=0\in P_2$ and so $x\in P_2$, which is a contradiction. Therefore, $P_1\cap P_2=P_1P_2=0$ which implies that $A\cong A/P_1 \times A/P_2$ which is a direct product of two fields by Chinese remainder theorem.
\end{proof}
Let $E$ be an $A$-module. The idealization $A(+)E=\{(a,m):a\in A,m\in E\}$ of $E$ is a commutative ring with the usual component-wise addition and scalar multiplication defined by $(a_1,m_1)(a_2,m_2)=(a_1a_2,a_1m_2+a_2m_1)$ for every $a_1,a_2\in A$ and $m_1,m_2\in E$, see \cite{Anderson}.
\begin{theorem}\label{idealization}
Let $E$ be an $A$-module. If $A$ is a field then $A(+)E$ is a quasi second ring. In particular, the converse holds if $E$ is a factorial $A$-module.
\end{theorem}
\begin{proof}
Note that since $A$ is a field, then $A(+)E$ is a local ring with the maximal ideal $0(+)E$ by \cite[Theorem 3.2]{Anderson}. Therefore, $A(+)E$ is a quasi second ring as $(0(+)E)^2=0(+)0$ by Theorem \ref{qsr}. For the converse part, we first show that $A$ is a quasi second ring. Let $a,b\in A$ such that $0\neq aA\subseteq bA\neq A$. Since $A$ is a factorial module then $0\neq aE\subseteq bE\neq E$ if and only if $aA\subseteq bA$ by \cite[Theorem 2.1]{Lu}. Now, consider the following chain $0(+)0\neq (a,0)A(+)E=aA(+)aE\subseteq (b,0)A(+)E\neq A(+)E$. It follows that $aA=bA$, so $A$ is a quasi second ring. Moreover, since $A$ is a UFD then by Theorem \ref{qsr}, $(A,\mathfrak{m})$ is a local ring and $\mathfrak{m}^2=0$. Hence, $\mathfrak{m}=0$ which implies that $A$ is a field as $A$ is an integral domain. 
\end{proof}
The next example highlights the condition of factorial module in Theorem \ref{idealization}.
\begin{example}
Consider the $A=K[x,y]/(x,y)^2$-module $E=(x,y)/(x,y)^2$ where $K$ is a field. One can see that $E$ is not a factorial module and $A$ is a local ring with the unique maximal ideal $P=(x,y)/(x,y)^2$ with $P^2=0$ and $PE=0$. Note that $(P(+)E)^2=0(+)0$ which implies that $R(+)E$ is a quasi second ring by Theorem \ref{qsr}.
\end{example}
\begin{theorem}\label{T1-starcondition}
$qD_A(E)$ is a $T_1$-space if and only if $E$ is a quasi second $A$-module.
\end{theorem}
\begin{proof}
$\Rightarrow:$ Suppose $qD_A(E)$ is a $T_1$-space and let $0\neq bE\subseteq aE\neq E$ for some $a,b\in A$. Note that $[b]\in \overline{\{[a]\}}=\{[a]\}$ by Proposition \ref{Closure}. This implies that $[b]=[a]$ and so $bE=aE$. Thus, $E$ is a quasi second $A$-module.

$\Leftarrow:$ Let $[b]\in \overline{\{[a]\}}$ for some $a,b\in W(E)^\#$. Then, $0\neq bE\subseteq aE\neq E$ which implies $bE=aE$ and so $[b]=[a]$ by the assumption. Hence, $\overline{\{[a]\}}=\{[a]\}$ i.e. $qD_A(E)$ is a $T_1$-space.
\end{proof}

\begin{theorem}\label{||=2}
Let $E$ be a finitely generated comultiplication module such that $\ann(E)$ is a semiprime ideal of $A$. Then, $qD_A(E)$ is a $T_1$-space if and only if $qD_A(E)$ is empty or discrete space with $|\Ec(W(E)^\#)|=2$.
\end{theorem}
\begin{proof}
In view of Theorem \ref{SimpleorWeakIdempotent}, we have the following two cases. If $E$ is a simple module, then $qD_A(E)$ is an empty space. For the other case, we have $xE=eE$ or $xE=(1-e)E$ for every $x\in W(E)^\#$ and for some weak idempotent $e\in A$ such that $\bar e\neq\bar0,\bar1$. It follows that $\Ec(W(E)^\#)=\{[e],[1-e]\}$ together with $U_e=\{[e]\}$ and $U_{1-e}=\{[1-e]\}$ since $(1-e)E\not\subseteq eE$ and $eE\not\subseteq (1-e)E$. Therefore, $qD_A(E)$ is a discrete space with $|\Ec(W(E)^\#)|=2$.
\end{proof}
The next proposition will be used in the following examples to produce examples of (in)finite $T_1$-spaces.
\begin{proposition}\label{ainW(M)iff}
Let $E$ be a finitely generated $A$-module. Then, $a\in W(E)^\#$ if and only if $a+\ann(E)$ is a nonzero nonunit element in $A/\ann(E)$.
\end{proposition}
\begin{proof}
Let $a\in W(E)^\#$ then $0\neq aE\neq E$. If $\bar a \in A/\ann(E)$ is a unit element, then $(\bar a)=A/\ann(E)\iff ((a)+\ann(E)) / \ann(E)= A/\ann(E)$. This implies that $aE=E$, which is a contradiction. Conversely, let $\bar a \in A/\ann(E)$ be a nonzero nonunit element then $aE\neq 0$. Now, suppose $aE=E$ then $1-ra \in \ann(E)$ if and only if $\bar a\in A/\ann(E)$ is a unit element by Nakayama's Lemma. Therefore, $a\in W(E)^\#$ as described in the proposition. 
\end{proof}
\begin{example}\label{infiniteT1space}
Consider the $A=\Z_2[\{x_i\}_{i\in I}]$-module $E=\Z_2[\{x_i\}_{i\in I} ]/(\{x_i\}_{i\in I})^2$ for any index set $I$. Note that $A$-module $E$ is a finitely generated, indeed a cyclic module. Moreover, one can explicitly compute that $\ann(E)=(x_i:i\in I)^2$ and $(A/\ann(E) )^\#=\left\{\sum_{i\in I} a_ix_i:a_i\in \Z_2\right\}\setminus \Z_2$ by Proposition \ref{ainW(M)iff} where $A/\ann(E)$ is a ring with $2^{|I|+1}$ elements if $I$ is a finite index set. Hence, there are infinitely many equivalence classes in $\Ec(W(E)^\#)$ that satisfy $aE=\{\bar 0,\bar a\}$ for all $a\in W(E)^\#$. By Proposition \ref{T1-starcondition}, $qD_A(E)$ is an infinite $T_1$, indeed an infinite discrete space.
\end{example}
\begin{example}
Consider the $A=\Z_2[x,y,z]$-module $E=\Z_2[x,y,z] /(x,y,z)^2$ as defined in Example \ref{infiniteT1space}. In a similar fashion, one can compute that $\ann(E)=(x,y,z)^2$ and $A/\ann(E)$ is a ring with $16$ elements for which $(A/\ann(E))^\#=\{\overline x,\overline y,\overline z,\overline{x+y},\overline{x+z},\overline{y+z},\overline{x+y+z}\}$ by Proposition \ref{ainW(M)iff}. Moreover, there are no pairs $a$ and $b$ in $ W(E)^\#$ satisfying $aE\subseteq bE$. Hence, $qD_A(E)$ is a $T_1$-space with $|\Ec(W(E)^\#|=7$ by Proposition \ref{T1-starcondition}.
\end{example}
The proofs of the following theorem are omitted, as they are straightforward.
\begin{theorem}\label{dsT1} 
Let $E$ be an $A$-module. Then the following statements hold:
\begin{enumerate}
\item[i)] $qD_A(E)$ is discrete space if and only if $aE$ is maximal element of $\{bE\}_{b\in W(E)^\#}$ for all $a\in W(E)^\#$.
\item[ii)] $E$ is a quasi second $A$-module if and only if $qD_A(E)$ is a discrete space.
\end{enumerate}
\end{theorem}
We now conclude this section by showing that $qD_A(E)$ is a discrete space if and only if $E$ is a quasi second $A$-module. That is, $qD_A(E)$ satisfies each separation axiom if and only if $E$ is a quasi second $A$-module. Now, we recall the needed definitions for the next theorem. A topological space $X$ is called metrizable if $X$ is induced by a metric. The following theorem is an immediate consequence of Theorem \ref{dsT1} and the fact that Alexandrov $T_1$-spaces are discrete.
\begin{theorem}
Let $E$ be an $A$-module. Then the following statements are equivalent:
\begin{enumerate}
\item[i)] $E$ is a quasi second $A$-module
\item[ii)] $aE$ is a maximal element of $\{bE\}_{b\in W(E)^\#}$ for all $a\in W(E)^\#$.
\item[iii)] $qD_A(E)$ is a discrete space.
\item[iv)] $qD_A(E)$ is metrizable.
\item[v)] $qD_A(E)$ is a $T_2$-space.
\item[vi)] $qD_A(E)$ is a $T_1$-space.
\end{enumerate}
\end{theorem}
A topological space is called \textit{$T_3$-space} if there exist two disjoint open sets for each closed set $C$ and a point $x\in X-C$ \cite{Munkres}.
\begin{theorem}
Let $E$ be an $A$-module. Then, $qD_A(E)$ is a $T_3$-space if and only if $U_a$ is closed for all $a\in W(E)^\#$.
\end{theorem}
\begin{proof}
$\Rightarrow:$ Assume that $qD_A(E)$ is a $T_3$-space and choose $a\in W(E)^\#$. It follows that there exists an open set $V$ containing $X-U_a$ such that $V\cap U_a=\emptyset$. Moreover, note that $[b]\not\in U_a$ for all $[b]\in V$. Hence, $[b]\in X-U_a$ which implies that $V=X-U_a$, as desired.

$\Leftarrow:$ It sufficies to show that each open set is closed. Let $O$ be an open set then $O=\bigcup_{a\in W(E)^\#} U_a$ for some $a\in W(E)^\#$ by Proposition \ref{Prop1}. Then, $\overline{O}=\overline{\bigcup_{a\in W(E)^\#} U_a}=\bigcup_{a\in W(E)^\#} \overline{U_a}=\bigcup_{a\in W(E)^\#} U_a=O$ by Lemma \ref{unionofclosure}, as desired.
\end{proof}
Lastly, recall that a space $X$ is called a $T_5$-space if for any separated sets $A$ and $B$ there exist disjoint open sets containing them \cite{Munkres}. Moreover, an integral domain $A$ is called a \textit{valuation domain} if for any $a\in K$, either $a\in A$ or $a^{-1}\in A$ where $K$ is the quotient field of $A$. Equivalently, any ideal of $A$ is comparable with respect to inclusion, see \cite[Proposition 5.2]{Larsen}.
\begin{theorem} Let $E$ be an $A$-module. Then the following statements hold:
\begin{enumerate}
\item[i)] If $A$ is a valuation domain then $qD_A(E)$ is a $T_5$-space.
\item[ii)] If $E$ is a torsion free nonuniserial and not a second module over $A$, then $qD_A(E)$ is not a $T_5$-space.
\end{enumerate}
\end{theorem}
\begin{proof}
\begin{enumerate}
\item[i)] One may have $bE\subseteq aE$ or $aE\subseteq bE$ for given $a,b\in W(E)^\#$ which implies that $[b]\in \overline{\{[a]\}}$ or $[a]\in \overline{\{[b]\}}$. Hence there are no separated sets in $qD_A(E)$ if $A$ is a valuation domain. 

\item[ii)] Let $a,b\in W(E)^\#$ then $aE\not\subseteq bE$ and $bE\not\subseteq aE$. Since $E$ is not a second module then there exists $c\in A$ such that $0\neq cE\neq E$, that is, $c\in W(E)^\#$. It follows that $\{[ac]\}$ and $\{[bc]\}$ are separated sets. Hence $qD_A(E)$ is not a $T_5$-space by observing that $[c]\in U_c\subseteq U_{ac}\cap U_{bc}$.
\end{enumerate}
\end{proof}
\section{Further Topological Properties of $qD_A(E)$}
The purpose of this section is to investigate several properties of $qD_A(E)$; namely (hyper,ultra)connectedness, lindelöf and compactness  together with countability axioms. We now briefly remind the definitions related to subsequent discussion.

A topological space $X$ is called \textit{ultraconnected} if there are no disjoint nonempty closed sets in $X$. Moreover, a space $X$ is called \textit{hyperconnected} if there are no disjoint nonempty open sets in $X$; equivalently, each open set in $X$ is dense, see \cite{Steen}.
\begin{proposition}\label{ultraconnected}
Let $E$ be an $A$-module and $\ann(E)$ is a prime ideal of $A$. Then $qD_A(E)$ is an ultraconnected space. In particular, $qD_A(E)$ is a $T_4$-space.
\end{proposition}
\begin{proof}
The proof is analogous to \cite[Proposition 8]{YiKo}.
\end{proof}
\begin{theorem} \label{hyperconnected}
$qD_A(E)$ is a hyperconnected space if and only if for any $a,b\in W(E)^\#$, there is $x\in W(E)^\#$ such that $aE+bE\subseteq xE\neq E$.
\end{theorem}
\begin{proof}
Suppose $qD_A(E)$ is hyperconnected then $U_a\cap U_b\neq\emptyset$ for given $a,b\in W(E)^\#$. Thus, one may choose $[x]\in U_a\cap U_b$  which satisfies that $aE,bE\subseteq xE$. Hence, $aE+bE\subseteq xE$. Conversely, let $O_1$ and $O_2$ be two nonempty disjoint open sets with $[a]\in O_1$ and $[b]\in O_2$. By assumption, there exists $x\in W(E)^\#$ so that $aE,bE\subseteq xE$. Therefore, $[x]\in O_1\cap O_2$ by Lemma \ref{smallest} which leads to a contradiction.
\end{proof}
\begin{corollary}
Let $E$ be an uniserial module over $A$. Then, $qD_A(E)$ is a hyperconnected space.
\end{corollary}
\begin{proof}
The proof is a direct consequence of Theorem \ref{hyperconnected}.
\end{proof}
The following example indicates that $E$ is not necessarily an uniserial module if $qD_A(E)$ is a hyperconnected space.
\begin{example}
Let $A$ be a local ring with a unique principal maximal ideal $P=(p)$. Consider the $A$-module $E=A\times A$. One can see that $E$ is not an uniserial module and for any pair $a,b\in W(E)^\#$, there exists $p\in A$ such that $aE+bE\subseteq pE\neq E$. Hence, $qD_A(A\times A)$ is a hyperconnected space by Theorem \ref{hyperconnected}.
\end{example}
A space $X$ is called \textit{first countable} or satisfies \textit{the first countability axiom} if each point of $X$ has a countable basis. Also, $X$ is called a \textit{second countable} space if $X$ has a countable basis, see \cite{Munkres}. Additionally, recall from \cite{Arenas} fact that Alexandrov $T_0$-space $X$ is second countable if and only if $X$ is countable. Moreover, $X$ is called a \textit{connected space} if there are no disjoint open sets that cover $X$. Also, for given $a,b\in X$, a continuous map $f:[c,d]\to X$ satisfying $f(c)=a$ and $f(d)=b$ is called a \textit{path} in $X$ from $a$ and $b$. A space $X$ is called \textit{path connected} if there is a path between for each pair of $X$, see \cite{Munkres}. A topological space is called \textit{locally (path)connected} if there is a basis of $X$ containing (path)connected open sets, \cite{Steen}.
\begin{corollary} Let $E$ be an $A$-module. Then, 
\begin{enumerate}
\item[i)] $qD_A(E)$ satisfies first countability axiom.
\item[ii)] $qD_A(E)$ is (path)connected space if $\ann(E)$ is a prime ideal.
\item[iii)] $qD_A(E)$ is a second countable space if and only if $\Ec(W(E^\#))$ has countably many elements or $A$ is a countable integral domain.
\item[iv)] $qD_A(E)$ is a locally (path)connected space.
\end{enumerate}
\end{corollary}
\begin{proof}
The proofs follow directly from Lemma \ref{smallest}, Proposition \ref{ultraconnected} and \cite[Theorem 2.8]{Arenas}.
\end{proof}
\begin{example}
We remark that if $A$ is a countable integral domain then $\Ec(W(E)^\#)$ is countable but the converse is not generally true. For example, $\Ec(W(K)^\#)=\emptyset$
for $K$-module $K$ where $K$ is an uncountable field.
\end{example}
We say that a topological space $X$ is \textit{Lindelöf} (resp. \textit{compact}) if every open cover has a countable (finite) subcover.
\begin{theorem}\label{Lindelöf}
$qD_A(E)$ is Lindelöf if and only if the family $\{aE\}_{a\in W(E)^\#}$ has countably many minimal elements with respect to inclusion.
\end{theorem}
\begin{proof}
$\Rightarrow:$ Suppose $qD_A(E)$ is a Lindelöf space then there is a countable index set $\triangle$ such that $\Ec(W(E)^\#)=\displaystyle\bigcup_{i\in \triangle} U_{a_i}$ by Proposition \ref{Prop1}. Now, choose $bE\in \{aE\}_{a\in W(E)^\#}$ then $[b]\in \Ec(W(E)^\#)=\displaystyle\bigcup_{i\in \triangle} U_{a_i}$. It follows that there must be an index $i$ so that $[b]\in U_{a_i}$ which implies $a_iE\subseteq bE$. Hence, the minimal elements of $\{aE\}_{a\in W(E)^\#}$ must be one of the elements in $\{a_iE\}_{i\in \triangle}$.

$\Leftarrow:$ Let $\{a_iE\}_{i\in \triangle}$ be the set of minimal elements of $\{aE\}_{a\in W(E)^\#}$ for some countable index set $\triangle$. Let $\Ec(W(E)^\#)=\displaystyle\bigcup_{j\in \Gamma} O_j$ for some open sets $O_j$ in $qD_A(E)$. Moreover, one can express $\Ec(W(E)^\#)$ as $\Ec(W(E)^\#)=\displaystyle\bigcup_{j\in \Gamma} O_j=\bigcup_{j\in \Gamma} U_{x_j}$ for some $x_j\in W(E)^\#$. Note that $[a_i]\in \Ec(W(E)^\#)=\bigcup_{j\in \Gamma} U_{x_j}$ for every $i\in \triangle$. Hence, there must be an index $x_{j_i}$ such that $[a_i]\in U_{x_{j_i}}$ which implies that $x_{j_i}E\subseteq a_iE$. Thus, $U_{x_{j_i}}=U_{a_i}$ by the minimality of $a_iE$. Also, since $U_a\subseteq U_{a_i}$ for any $a\in W(E)^\#$ one can write the following expression $\displaystyle \Ec(W(E)^\#)=\bigcup_{a\in W(E)^\#} U_a=\bigcup_{i\in \triangle} U_{a_i}=\bigcup_{i\in \triangle} U_{x_{j_i}}$ for countable index set $\triangle$. Therefore, $qD_A(E)$ is a Lindelöf space.
\end{proof}
In a similar fashion, the previous theorem can be easily modified for compact spaces.
\begin{theorem}
$qD_A(E)$ is compact if and only if the family $\{aE\}_{a\in W(E)^\#}$ has finitely many minimal elements with respect to inclusion.
\end{theorem}
\begin{proof}
The proof is analogous to Theorem \ref{Lindelöf}.
\end{proof}
In the following, we provide an example of a topological space having finitely many minimal elements in the set $\{aE\}_{a\in W(E)^\#}$, however $qD_A(E)$ is not still a compact space.
\begin{example}
Consider the $A=\Z$-module $E=\Z$. In this case, the set $\{aE\}_{a\in W(E)^\#}$ has no minimal elements; however, $qD_\Z(\Z)$ is not a compact space.
\end{example}
We conclude the paper with an application of quasi divisor topology which leads to an alternative way to show that localization of $\Z$ at the prime ideal $P$ has infinitely many irreducible elements.
\begin{theorem}\label{infinitudeofirreducibles}
Let $E$ be a factorial module over an integral domain $A$ which is not a field. If $|S|<|A|$ then $A$ has infinitely many irreducible elements.
\end{theorem}
\begin{proof}
Suppose $A$ has finitely many irreducibles in $A$, namely $a_1,\dots,a_n$. It follows that $S$ is finite since there are countably many nonzero nonunit elements together with countably many units by the following injection $f:u(A)\hookrightarrow A^\star$ where $A^\star$ is the set of nonzero nonunit elements of $A$. Moreover, $\operatorname{Irr(W(E)^\#)=Irr(S)}$ as $E$ is a factorial module. Thus, $\Ec(\Irr(W(E^\#))=\Ec(\Irr(S))$ is dense in $qD_A(E)$ by Corollary \ref{EC(S)dense}. Hence, one can conclude that $\overline{\Ec(\Irr(A))}=\bigcup_{i=1}^n \overline{\{[a_i]\}}=\Ec(W(E)^\#)$, particularly $\bigcap_{i=1}^n \overline{\{[a_i]\}}^c =\emptyset$. Now, set $x_m=a_1^m+a_2\dots a_n$ for each $m\in \mathbb N$. This implies that there must be an index $t\in \mathbb N$ such that $x_t\not\in S$. However, we claim that $x_t\in \overline{\{[a_i]\}}^c=\{[b]\in \Ec(W(E)^\#):a_iE\not\subseteq bE\}$. Otherwise, $a_iE\subseteq x_t E=(a_1^t+a_2\dots a_n)E$ implies that $a_1^t+a_2\dots a_n$ divides $a_i$ for each $i$ by \cite[Theorem 2.1]{Lu}, which is impossible. Thus, $A$ has infinitely many irreducible elements.
\end{proof}
\begin{example}
Let $\Z_{(p)}=\{ \frac ab: a,b\in \Z \text{ such that $p\nmid b$}\}$ be a localization of $\Z$ at the prime ideal $(p)$. Consider the factorial $E=\prod_{i\in I} \Z_{(p)}$-module $A=\Z_{(p)}$. Since any nonzero nonunit element of $\Z_{(p)}$ associates to $p$, one can see that $pE=\left(p^{k_0}\frac{a_0}{b_0},\dots, p^{k_i}\frac{a_i}{b_i},\dots \right)$ is maximal in the set $\{bE\}_{b\in W(E)^\#}$ where $k_i\ge 1$ and $(a_i,b_i)=1$. Therefore, $|S|=1<|A|$ which implies that $\Z_{(p)}$ has infinitely many irreducibles by Theorem \ref{infinitudeofirreducibles}.
\end{example}


\begin{thebibliography}{9}                         %
\bibitem {coz} M. Afkhami, K. Khashyarmanesh, The cozero-divisor graph of a commutative ring, Southeast Asian Bull. Math. 35 (2011), 753--762.

\bibitem {coz2} M. Afkhami, K. Khashyarmanesh, On the cozero-divisor graphs of commutative rings and their complements, Bull. Malays. Math. Sci. Soc. (2) 35(4) (2012), 935--944.

\bibitem{Anderson} D.D. Anderson, M. Winders, Idealization of a module, J. Commut. Algebra 1(1) (2009), 3--56.

\bibitem{Faranak} H. Ansari-Toroghy, F. Farshadifar, The dual notion of multiplication modules, Taiwanese J. Math. (2007), 1189--1201.

\bibitem{Faranak2} H. Ansari-Toroghy, F. Farshadifar, The Zariski topology on the second spectrum of a module, Algebra Colloq. 21(4) (2014), 683--698.

\bibitem {Arenas} F.G. Arenas, Alexandroff spaces, Acta Math. Univ. Comenianae 68(1) (1999), 17--25.

\bibitem{secil} Çeken, S. (2022). On S-second spectrum of a module, Revista de la Real Academia de Ciencias Exactas, Fisicas y Naturales, Serie A. Matematicas, 116(4), 171.

\bibitem{Smith2} Z.A. El-Bast, P.F. Smith, Multiplication modules, Commun. Algebra 16(4) (1988), 755--779.

\bibitem {FacSal} A. Facchini, L. Salce, Uniserial modules: sums and isomorphisms of subquotients, Commun. Algebra 18(2) (1990), 499--517.

\bibitem{Hartshorne} R. Hartshorne, Algebraic geometry, Vol. 52, Springer Science \& Business Media, 2013.

\bibitem{vn} C. Jayaram, Ü. Tekir, von Neumann regular modules, Commun. Algebra 46(5) (2018), 2205--2217.

\bibitem{JTK} C. Jayaram, Ü. Tekir, S. Koç, On Baer modules, Rev. Unión Mat. Argent. 63(1) (2022), 109--128.

\bibitem{Larsen} M. Larsen, P. McCarthy, Multiplicative Theory of Ideals, Academic Press, 1971.

\bibitem{Lombardi} H. Lombardi, C. Quitté, Commutative algebra: constructive methods, finite projective modules, Vol. 20, Springer, 2015.

\bibitem {Lu} C.P. Lu, Factorial modules, Rocky Mountain J. Math. 7(1) (1977), 125--139.

\bibitem{Lu2} C.P. Lu, The Zariski topology on the prime spectrum of a module, Houston J. Math. 25 (1999), 417--432.

\bibitem{McCasland} R.L. McCasland, M.E. Moore, P.F. Smith, An introduction to Zariski spaces over Zariski topologies, Rocky Mountain J. Math. (1998), 1357--1369.

\bibitem {Munkres} J. Munkres, Topology, 2nd ed., Prentice Hall, Upper Saddle River, 2000.

\bibitem{Richmond} T. Richmond, General topology: an introduction, Walter de Gruyter GmbH \& Co KG, 2020.

\bibitem{Smith} P.F. Smith, Some remarks on multiplication modules, Arch. Math. 50(3) (1988), 223--235.

\bibitem {Steen} L.A. Steen, J.A. Seebach, Counterexamples in topology, Vol. 18, Springer, New York, 1978.

\bibitem{Tekir} Ü. Tekir, The Zariski topology on the prime spectrum of a module over noncommutative rings, Algebra Colloq. 16(4) (2009), 649--658.

\bibitem {Yas} S. Yassemi, The dual notion of prime submodules, Arch. Math. (Brno) 37(4) (2001), 273--278.

\bibitem{Eda} E. Yıldız, et al., On S-Zariski topology, Commun. Algebra 49(3) (2021), 1212--1224.

\bibitem {YiKo} U. Yiğit, S. Koç, On divisor topology of commutative rings, Ric. Mat. (2025), 1--13.

\end{thebibliography}
\end{document}